\documentclass[letterpaper,12pt]{amsart}

\usepackage[abbrev]{amsrefs}

\usepackage[T1]{fontenc}
\usepackage[english]{babel}
\usepackage{amsthm}
\usepackage{cases}
\usepackage{amsfonts}
\usepackage{amsmath}
\usepackage{amssymb}

\usepackage{mathtools}
\usepackage{enumerate}
\usepackage{chngcntr}
\usepackage{hyperref}
\counterwithin{equation}{section}
\usepackage{float}
\calclayout

\usepackage{enumitem}
\title{Integral points on cubic twists of Mordell curves}
\author{Stephanie Chan}
\date{\today}

\address{Department of Mathematics, University of Michigan}
\email{ytchan@umich.edu}

\newcommand{\QQ}{\mathbb{Q}}
\newcommand{\ZZ}{\mathbb{Z}}
\newcommand{\leg}[2]{\left(\frac{#1}{#2}\right)}

\DeclareMathOperator{\GL}{GL}

\newtheorem{theorem}{Theorem}[section]
\newtheorem{corollary}[theorem]{Corollary}
\newtheorem{lemma}[theorem]{Lemma}
\newtheorem{remark}[theorem]{Remark}

\newtheorem{proposition}[theorem]{Proposition}

\begin{document}
\begin{abstract}Fix a non-square integer $k\neq 0$.
We show that the number of curves $E_B:y^2=x^3+kB^2$ containing an integral point, where $B$ ranges over positive integers less than $N$, is bounded by $O_k(N(\log N)^{-\frac{1}{2}+\epsilon})$. In particular, this implies that the number of positive integers $B\leq N$ such that $-3kB^2$ is the discriminant of an elliptic curve over $\QQ$ is $o(N)$. The proof involves a discriminant-lowering procedure on integral binary cubic forms. \end{abstract}

\maketitle

\section{Introduction}
Fixing a non-zero integer $k$, Mordell's equation
\[E:y^2=x^3+k\]
defines an elliptic curve over $\QQ$.
We consider the cubic twist family of $E$, which contains elliptic curves defined by the Weierstrass equation
\[E_B:y^2=x^3+kB^2,\] 
where $B$ ranges over positive integers.
The set of integral points is defined as
\[E_B(\ZZ)\coloneqq \{(x,y)\in\ZZ^2:y^2=x^3+kB^2\}.\]
Ordering $E_B$ by the size of $B$, we are interested in counting the number of curves which has an integral point.

\begin{theorem}\label{theorem:curvebd}
Fix a non-zero integer $k$ that is not a rational square. We have
\[\#\{1\leq B\leq N:E_B(\ZZ)\neq \varnothing\}\ll_k N\left(\frac{\log\log N}{\log N}\right)^{\frac{1}{2}}.\]
\end{theorem}

Theorem~\ref{theorem:curvebd} shows that, when $k\neq 0$ is not a square, almost all elliptic curves in this family $E_B$ does not have any integral points. We are also interested in bounding the average number of integral points in this family. Alpoge~\cite{Alpoge} proved that the average number of integral points is bounded in the family of all minimal short Weierstrass equations of elliptic curves, as well as in several other special families. In this direction, we show that the average number of integral points is also bounded in this family $E_B$ when we restrict $B$ to cube-free integers, so that $E_B$ come from distinct isomorphism classes over $\QQ$.
\begin{theorem}\label{theorem:ptbd}
Fix a non-zero integer $k$ that is not a rational square.  We have
\[\sum_{\substack{1\leq B\leq N\\ B\text{ cube-free}}}\#E_B(\ZZ)\ll_k N.\]
\end{theorem}
It is known that the average number of integral points is bounded in the family of Mordell curves $y^2=x^3+k$. Indeed, from a special case of~\cite[Theorem~1(a)]{EvertseSilverman}, the number of integral points on $y^2=x^3+k$ is $\ll h_3(-k)$, where $h_3$ is the $3$-part of the class number of $\QQ(\sqrt{-k})$. Then the boundedness of the average follows from the work of Davenport and Heilbronn~\cite{DH2}.
Note that however this does not imply Theorem~\ref{theorem:ptbd}, since we are dealing with a much thinner family.

To prove Theorem~\ref{theorem:curvebd} and Theorem~\ref{theorem:ptbd}, we will make use of the correspondence by Mordell~\cite{Mordell} between integer solutions of $y^2=x^3+kB^2$ and certain integral binary cubic forms. The main strategy is to carry out a discriminant-lowering procedure on the set of the integral binary cubic forms.

Observe that the bound as stated in Theorem~\ref{theorem:curvebd} cannot hold when $k$ is square, since the $3$-torsion points $(0,\pm \sqrt{k}B)\in E_B(\QQ)$ would be integral points for all $B$. One might expect a similar bound to hold after excluding such torsion points, but our method, which relies on the fact that very few integers are represented by the quadratic form $x^2-ky^2$ to achieve a saving, does not seem to be adaptable to this particular case. 

We will give some heuristics in Section~\ref{section:heuristics}, which suggest that the orders of magnitude of the quantities in Theorem~\ref{theorem:curvebd} and Theorem~\ref{theorem:ptbd} should be $\asymp_k N^{\frac{2}{3}}$, though when $k$ is a square the points $(0,\pm \sqrt{k}B)$ should be removed.

Counting integer solutions to Mordell equation is related to counting the number of elliptic curves over $\QQ$ with bounded discriminants.  
Given the Weierstrass equation of an elliptic curve, the standard invariants $c_4$ and $c_6$ give integral solutions to the equation
\[c_4^3-c_6^2 = 1728\Delta,\]
where $\Delta$ is the discriminant of an elliptic curve~\cite[Chapter~III.1]{Silverman}.
 Brumer and McGuinness~\cite{BrumerMcGuinness} conjectured that the number of elliptic curves with 
 $ \lvert \Delta \rvert\leq N$ is $\sim cN^{\frac{5}{6}}$ for an explicit constant $c> 0$ (See also~\cite{Watkins}), but currently even an upper bound of $o(N)$ seems out of reach~\cite{Young}.  
 
An immediate consequence of Theorem~\ref{theorem:curvebd} is that, for any fixed squarefree integer $k\notin\{0,-3\}$, the
integers that are discriminants of elliptic curves are zero-density among all integers of the form $kB^2$. 
\begin{corollary}
Let $k\notin -3\cdot (\QQ^\times)^2$ be a non-zero integer.
The number of positive integers $B\leq N$ such that $kB^2$ is the discriminant of an elliptic curve over $\QQ$ is
\[\ll_k N\left(\frac{\log\log N}{\log N}\right)^{\frac{1}{2}}.\]
\end{corollary}

\section{Heuristics}\label{section:heuristics}
We can obtain heuristic estimates for $\sum_{1\leq B\leq N}\#E_B(\ZZ)$ similar to~\cite{BrumerMcGuinness} by replacing lattice point counts by area computations.
From the equation $y^2=x^3+kB^2$, we want to count $(x,y)\in\ZZ^2$ such that $\frac{1}{k}(y^2-x^3)$ is a square. 
Assuming that the density of squares near $\frac{1}{k}(y^2-x^3)$ is $(\frac{1}{k}(y^2-x^3))^{-\frac{1}{2}}$,
we can estimate the number of integral points heuristically by evaluating the integral
\[I_k\coloneqq\int\int_{0\leq \frac{1}{k}(y^2-x^3)\leq N^2}\frac{dxdy}{\sqrt{(y^2-x^3)/k}}.\]

If $k<0$, by rescaling the area, then integrating by parts
\begin{align*}
I_k&=\lvert k\rvert^{\frac{5}{6}}\int\int_{0\leq x^3-y^2\leq N^2}\frac{dxdy}{\sqrt{x^3-y^2}}\\
&=2\lvert k\rvert^{\frac{5}{6}}
\left(\int_{x>N^{\frac{2}{3}}}\int_{\sqrt{x^3-N^2}\leq y\leq x^{\frac{3}{2}}}\frac{dydx}{\sqrt{x^3-y^2}}
+\int_{x\leq N^{\frac{2}{3}}}\int_{0\leq y\leq x^{\frac{3}{2}}}\frac{dydx}{\sqrt{x^3-y^2}}\right)\\
&
=2\lvert k\rvert^{\frac{5}{6}}\left(\int_{x>N^{\frac{2}{3}}}\sin^{-1}\left(\frac{N}{x^{\frac{3}{2}}}\right)dx
+\frac{\pi}{2}\int_{x\leq N^{\frac{2}{3}}}dx\right)\\
&=3\lvert k\rvert^{\frac{5}{6}}N^{\frac{2}{3}}\int_{1}^{\infty}\frac{1}{\sqrt{u^3-1}}du
.
\end{align*}

If $k>0$,
proceeding similarly as before
\begin{align*}
I_k&=k^{\frac{5}{6}}\int\int_{0\leq y^2-x^3\leq N^2}\frac{dxdy}{\sqrt{y^2-x^3}}\\
&=2k^{\frac{5}{6}}
\left(\int_{x>0}\int_{x^{\frac{3}{2}}\leq y\leq \sqrt{x^3+N^2}}\frac{dydx}{\sqrt{y^2-x^3}}
+\int_{-N^{\frac{2}{3}}\leq x\leq 0}\int_{0\leq y\leq \sqrt{x^3+N^2}}\frac{dydx}{\sqrt{y^2-x^3}}\right)\\
&
=2k^{\frac{5}{6}}\int_{x\geq -N^{\frac{2}{3}}}\log \left(\frac{N+\sqrt{x^3+N^2}}{\lvert x\rvert^{\frac{3}{2}}}\right)dx\\
&=3k^{\frac{5}{6}}N^{\frac{2}{3}}\int_{-1}^{\infty}\frac{1}{\sqrt{u^3+1}}du
.\end{align*}

The size of $I_k$ seems to suggest that $\sum_{1\leq B\leq N}\#E_B(\ZZ)$ should be of size $\asymp \lvert k\rvert^{\frac{5}{6}}N^{\frac{2}{3}}$. However, this does not make sense when $k$ is a square because of the points $(0,\pm \sqrt{k}B)\in E_B(\ZZ)$. We expect that the correct estimate is instead $\sum_{1\leq B\leq N}\#E^*_B(\ZZ)\asymp \lvert k\rvert^{\frac{5}{6}}N^{\frac{2}{3}}$, where 
\[E^*_B(\ZZ)\coloneqq\begin{cases}
E_B(\ZZ)\setminus\{(0,\pm \sqrt{k}B)\}&\text{ if }k\text{ is a square},\\
E_B(\ZZ)&\text{ otherwise.}
\end{cases}\] 

We can get a lower bound that match with the expected order of $\sum_{1\leq B\leq N}\#E^*_B(\ZZ)$ in $N$ for fixed $k$, by direct counting. 
Take $0<b\leq\frac{1}{2}N^{\frac{1}{3}}k^{-\frac{1}{3}}$ and $0<\lvert d\rvert\leq\frac{1}{2}N^{\frac{1}{3}}k^{\frac{1}{6}}$.
Then $(d^2-kb^2,d(d^2-kb^2))\in E^*_B(\ZZ)$, where $B=\lvert b(d^2-kb^2)\rvert\leq N$.
Counting the pairs of $(d,b)$ gives 
\[\sum_{1\leq B\leq N}\#E^*_B(\ZZ)\geq \frac{1}{2}k^{-\frac{1}{6}}N^{\frac{2}{3}}.\]

A similar lower bound can be obtained for
$\#\{1\leq B\leq N:E^*_B(\ZZ)\neq \varnothing\}.$
Using~\cite[Theorem~1.1]{StewartXiao} (or~\cite[Theorem~1.1]{StewartXiaokfree} if we require $B$ to be cube-free), we can count the number of non-zero integers between $-N$ and $N$ that are represented by the binary cubic form $y(x^2-ky^2)$. Whenever $B=\lvert y(x^2-ky^2)\rvert$, there is an integral point $(x^2-ky^2,x(x^2-ky^2))\in E^*_B(\ZZ)$. This gives the estimate
\[\#\{1\leq B\leq N:E^*_B(\ZZ)\neq \varnothing\}\gg k^{-\frac{1}{6}}N^{\frac{2}{3}},\]
when $N$ is large enough relative to $k$.

\section{Binary cubic forms}
\subsection{Invariants and covariants}
We first recall some facts about the invariants and covariants of binary cubic forms. We refer the reader to~\cite[Section~3]{Cremona} for an overview of these quantities. Note that the formulas stated here are scaled differently, as we will only work with \emph{integer-matrix binary cubic forms}, namely forms of the shape
\[ax^3+3bx^2y+3cxy^2+dy^3, \text{ where }a,b,c,d\in\ZZ.\]

Let $f(x,y)=ax^3+3bx^2y+3cxy^2+dy^3$ be an integer-matrix binary cubic forms. 
The discriminant of $f$ is defined to be
\[\Delta=\Delta(f)\coloneqq 3b^2c^2-4 ac^3-4 b^3d-a^2d^2+6abcd.\]
The seminvariants of $g$ are $\Delta$, the leading coefficient $a$, 
\[H\coloneqq b^2-ac \qquad\text{ and }\qquad U\coloneqq 2b^3+a^2d-3abc.\]
The seminvariant $H$ is the leading coefficient of the Hessian covariant
\[H(x,y)\coloneqq (b^2-ac)x^2+(bc-ad)xy+(c^2-bd)y^2.\]
The seminvariants are related by the following syzygy
\begin{equation}\label{eq:syzygy}
\left(\frac{1}{2}U\right)^2=P^3-\frac{1}{4}\Delta a^2.
\end{equation}

Let $V$ be the space of integer-matrix binary cubic form.
Given any $f\in V$ and $(x_0,y_0)\in\ZZ^2$, define the action of $\gamma\in\GL_2(\ZZ)$ on the pair $(f,(x_0,y_0))$ by
\[\gamma\cdot (f(x,y),(x_0,y_0))=(f((x,y)\cdot \gamma),(x_0,y_0)\cdot\gamma^{-1} ).\]
Notice that this action preserves the value of $f(x_0,y_0)$.
\subsection{Relation to integral points}
We will make use of the correspondence between certain integral binary cubic forms and the integer solutions of $y^2=x^3+kB^2$ due to Mordell. See for example~\cite[Chapter~24]{Mordell} for a description of the correspondence, or~\cite{BGMordell} for a recent study of this correspondence in enumerating integral points.
More explicitly, given any $P=(c,d)\in E_B(\ZZ)$, we can attach to $P$ an binary integral cubic form 
\[f_P(x,y)=x^3-3cxy^2+2dy^3,\]
which has discriminant $\Delta(f_P)=-4kB^2$.

The correspondence can be formulated as follows. See~\cite[Section~2]{BGMordell} for the proof for general Mordell curves.
\begin{lemma}\label{lemma:Mcorr}
Fix integers $k\neq 0$ and $B\neq 0$. The following sets are in bijections:
\begin{enumerate}
\item $E_B(\ZZ)$;
\item binary cubic forms of the shape $f(x,y)=x^3-3cxy^2+2dy^3$, where $c,d\in \ZZ$ and $\Delta(f)=-4kB^2$; and
\item $\GL_2(\ZZ)$-equivalence classes of $(f,(x_0,y_0))\in V\times \ZZ^2$ satisfying $f(x_0,y_0)=1$ and $\Delta(f)=-4kB^2$.
\end{enumerate}
The forward map is given by
\[P\mapsto f_P\mapsto (f_P,(1,0)).\]
\end{lemma}
\subsection{Reducible forms}
We wish to count the number of integral points that correspond to reducible forms under Lemma~\ref{lemma:Mcorr}. We will make use of some ideas in~\cite[Section~3.3]{BGMordell} to prove Lemma~\ref{lemma:reducibleform} and Lemma~\ref{lemma:reducible}.

\begin{lemma}\label{lemma:reducibleform}
Fix $\epsilon>0$ and an integer $k\neq 0$.
The number of $\GL_2(\ZZ)$-equivalence classes of integer-matrix binary cubic forms $f$ that satisfies the properties:
\begin{enumerate}
\item $f$ is reducible,
\item $f(x_0,y_0)=1$ for some $x_0,y_0\in\ZZ$; and
\item $\Delta(f)=-4kB^2$ for some positive integer $B\leq N$,
\end{enumerate}
is at most $O_k \left(N^{\frac{2}{3}}(\log N)^{1+\epsilon}\right)$.
\end{lemma}
\begin{proof}
Without loss of generality, assume $k$ is squarefree.
If $f$ satisfies the listed properties, then under $\GL_2(\ZZ)$-equivalence, we can assume
\begin{equation}\label{eq:reducibleg}
f(x,y)=x(x^2+3bxy+3cy^2),
\end{equation}
where $b,c\in\ZZ$. 
The discriminant of $f$ is $\Delta(f)=c^2(3b^2-4c)=-4kB^2$, which we can rewrite as 
\begin{equation}\label{eq:cb}
12c=(3b)^2+3k\left(\frac{2B}{c}\right)^2.
\end{equation}
Since $k$ is assumed to be squarefree, $c\mid 2B$. Also $c\neq 0$ since $\Delta(f)\neq 0$.

First consider when $\lvert 2B/c \rvert \leq  N^{\frac{1}{3}}(\log N)^{\frac{1}{2}-\epsilon}$ and $\lvert b\rvert >  N^{\frac{1}{3}}(\log N)^{\frac{1}{2}+2\epsilon}$.
From~\eqref{eq:cb}, we see that $c\gg_k b^2$, and so
$2B/c\ll_k N/b^2$ and $b\ll_k N^{\frac{1}{2}}$.
Since $c$ is determined by $b$ and $2B/c$ using~\eqref{eq:cb}, it suffices to count the possible pairs of $(b,2B/c)$.
Then the number of pairs $(b,2B/c)$ in this range is bounded by
\[\ll_k\sum_{N^{\frac{1}{3}}(\log N)^{\frac{1}{2}+2\epsilon}<b\ll_k N^{\frac{1}{2}}}\frac{N}{b^2}=N^{\frac{2}{3}}(\log N)^{-\frac{1}{2}-2\epsilon}.\]

If $\lvert 2B/c\rvert \leq  N^{\frac{1}{3}}(\log N)^{\frac{1}{2}-\epsilon}$ and $\lvert b\rvert \leq N^{\frac{1}{3}}(\log N)^{\frac{1}{2}+2\epsilon}$, then the number of pairs $(b,2B/c)$ in this range is $\ll N^{\frac{2}{3}}(\log N)^{1+\epsilon}$ as required.

The remaining case is when $\lvert 2B/c\rvert >N^{\frac{1}{3}}(\log N)^{\frac{1}{2}-\epsilon}$, so \[\lvert c\rvert <2N^{\frac{2}{3}}(\log N)^{-\frac{1}{2}+\epsilon}.\]
This is not possible if $k$ is positive because $12c<(2B/c)^2$, so assume that $k$ is negative.
From~\eqref{eq:cb}, we have the factorisation
\[\left(3b-\sqrt{-3k}\cdot\frac{2B}{c}\right)\left(3b+\sqrt{-3k}\cdot\frac{2B}{c}\right)=12c.\]
If $k\neq -3$, for each $c$, the number of factorisations of $12c$ as two conjugate ideals in $\QQ(\sqrt{-k})$ is $\ll \prod_{p\mid c,\ (\frac{-3k}{p})=1}(v_p(c)+1)$. Since 
\[3b+\sqrt{-3k}\cdot \frac{2B}{c}\ll_k N,\]
the possible generator given each principle ideal in a factorisation is $\ll_k\log N$ if $k\neq -3$. When $k=-3$, there are $\ll \prod_{p\mid c}(v_p(c)+1)$ factorisations and exactly one pair of $(2B/c, b)$ from each factorisation.
Therefore fixing any square-free $k$, the number of $(2B/c, b)$ in this case is bounded by 
\begin{align*}
&\ll_k
\begin{cases}
\displaystyle\sum_{\lvert c\rvert <2N^{\frac{2}{3}}(\log N)^{-\frac{1}{2}+\epsilon}} \prod_{ p}2^{v_p(c)}  &\text{if }k= 1,\vspace{1em}
\\
\displaystyle\log N\cdot \sum_{\lvert c\rvert <2N^{\frac{2}{3}}(\log N)^{-\frac{1}{2}+\epsilon}} \prod_{ (\frac{-3k}{p})=1}2^{v_p(c)} &\text{if }k\neq 1
\end{cases}\\
&\ll
N^{\frac{2}{3}}(\log N)^{1+\epsilon}.\end{align*}
\end{proof}
\begin{lemma}\label{lemma:reducible}
Let $\epsilon>0$,
We have
\[\sum_{0\leq B\leq N}\#\left\{P\in E_B(\ZZ):f_P\text{ is reducible}\right\}\ll_k N^{\frac{2}{3}}(\log N)^{1+\epsilon}.\]
\end{lemma}
\begin{proof}
If $f_P$ is a reducible form, then $f_P$ is $\GL_2(\ZZ)$-equivalent to a form
\[g(x,y)=x(x^2+3bxy+3cy^2),
\]
where $b,c\in\ZZ$. Write $(g(x,y),(x_0,y_0))=\gamma\cdot (f_P(x,y),(1,0))$.
Then since $f_P(1,0)=1$, we have
$g(x_0,y_0)=1$.
From the factorisation of $x_0(x_0^2+3bx_0y_0+3cy_0^2)=1$, we see that $x_0=\pm 1$. Putting back $x_0=\pm 1$ to the equation allows us to solve for $y_0$ in terms of $b$ and $c$.
Therefore $(f_P,(1,0))$ is $\GL_2(\ZZ)$-equivalent to one of $(g,(1,0))$, $(g,(1,-b/c))$, and $(g,(-1,(3b\pm\sqrt{3(3b^2-4c)})/6c))$. Hence each $\GL_2(\ZZ)$-class of reducible forms can only correspond to at most $4$ integral points under Lemma~\ref{lemma:Mcorr}.

Applying Lemma~\ref{lemma:reducibleform} to bound the number of possible $g$ completes the claim.
\end{proof}

\section{Lowering the discriminant}
We start by showing that we can transform $f_P$ into integral cubic forms with smaller discriminants. A similar discriminant-lowering lemma for binary quartic forms can be found in~\cite{ChanCongruent}.
\begin{lemma}\label{lemma:reddisc}
Let $P=(c,d)\in E_B(\ZZ)$. 
Take a positive integer $M$ dividing $B$ that is coprime to $c$.
Then there exists some integer $w\neq 0$ such that 
\begin{equation}\label{eq:discred}
F_P(x,y)\coloneqq \frac{1}{M^2}\cdot
f_P\left((x,y)\cdot \begin{pmatrix}
M & 0\\
w & 1
\end{pmatrix}\right)
\end{equation}
is an integer-matrix binary cubic form. Moreover, $F_P(1,0)=M$ and $\Delta(F)=-\frac{4B^2}{M^2}$.
\end{lemma}

\begin{proof}
Take $w\coloneqq c^{-1}d\mod M^2$.
We proceed to check that the coefficients of 
\[\begin{split}
f_P\left((x,y)\cdot \begin{pmatrix}
M & 0\\
w & 1
\end{pmatrix}\right)
&=f_P(Mx+wy,y)\\
&=M^3x^3+3M^2wx^2y+3M(w^2-c)xy^2+(w^3-3cw+2d)y^3
\end{split}\]
are all divisible by $M^2$.
The $x^3$- and $x^2y$-coefficients are clearly divisible by $M^2$.
Now put $w\coloneqq c^{-1}d\mod M^2$ into the $xy^2$- and $y^3$-coefficients and use $d^2=c^3+kB^2$, we see that \[w^2-c\equiv c^{-2}d^2-c=c^{-2}kB^2\bmod M^2\] and
\[w^3-3cw+2d\equiv c^{-3}d^3-3d+2d\equiv c^{-3}dkB^2\bmod M^2\]
are both divisible by $M^2$.
Therefore $F_P$ has integer coefficients. The remaining claims are immediate from the formula of $F_P$.
\end{proof}

Given an integral point $(c,d)\in E_B(\ZZ)$, take 
\[g\coloneqq\prod_{p\mid \gcd(c,B)} p^{v_p(B)},\] and apply Lemma~\ref{lemma:reddisc} with $M=B/g$. Take $F_P$ as in~\eqref{eq:discred}. The map
\[E_B(\ZZ)\rightarrow (V\times \ZZ^2)/\GL_2(\ZZ) \qquad
P\mapsto (F_P,(1,0))\]
is injective by the following lemma.
\begin{lemma}
Fix a positive integer $B$ and take $M\mid B$.
Suppose $P,Q\in E_B(\ZZ)$, and such that $\gcd(x(P),M)=\gcd(x(Q),M)=1$. Construct cubic forms $F_P$ and $F_Q$ with respect to $M$ via~\eqref{eq:discred}. 
If $(F_P,(1,0))$ and $(F_Q,(1,0))$ are $\GL_2(\ZZ)$-equivalent, then $P=Q$.
\end{lemma}
\begin{proof}
Take $w_P$ and $w_Q$ in place of $w$ in~\eqref{eq:discred} for $F_P$ and $F_Q$ respectively.
Suppose $\gamma\cdot(F_P,(1,0))=(F_Q,(1,0))$ for some $\gamma \in\GL_2(\ZZ)$. Then $(1,0)\cdot \gamma^{-1}=(1,0)$ implies that we can write $\gamma=\begin{pmatrix}
1 & 0\\
u & 1\end{pmatrix}$ for some $u\in\ZZ$.
From $F_P((x,y)\cdot \gamma)=F_Q(x,y)$ and
\[\begin{pmatrix}
1 & 0\\
u & 1\end{pmatrix}\cdot\begin{pmatrix}
M & 0\\
w_Q & 1\end{pmatrix}^{-1}\begin{pmatrix}
M & 0\\
w_P & 1\end{pmatrix}
=\begin{pmatrix}
1 & 0\\
uM+w_P-w_Q & 1\end{pmatrix},
\]
we have
\[f_P\left((x,y)\cdot \begin{pmatrix}
1 & 0\\
uM+w_P-w_Q & 1\end{pmatrix}\right)=
f_Q(x,y).
\]
The $x^2y$-coefficient of $f_P$ and $f_Q$ are both $0$, so we must have $f_P=f_Q$, and hence $P=Q$.
\end{proof}

\section{Counting integral points}
We now proceed to count the number of $\GL_2(\ZZ)$-equivalence classes of $(F_P,(1,0))$. 
We first transform the forms $F_P$ to reduced forms with bounded seminvariants.

\begin{proposition}[{\cite[Proposition~2, Proposition 4]{Cremona}}]\label{prop:reducedbd}
Every integer-matrix binary cubic form with discriminant $\Delta\neq 0$ is $\GL_2(\ZZ)$-equivalent to a reduced form with seminvariants in the following ranges
\[\lvert a\rvert\leq 2^{\frac{3}{2}}3^{-\frac{3}{4}}\lvert\Delta\rvert^{\frac{1}{4}},\qquad
\lvert H\rvert\leq 2^{\frac{1}{3}} 3^{-\frac{1}{2}}\lvert\Delta\rvert^{\frac{1}{2}}.\]
\end{proposition}

\begin{remark}
The possibility that $H=0$ should be included in~\cite[Proposition~4]{Cremona}. A particular example of a reduced cubic form with $H=0$ is $ax^3+by^3$ when $1\leq a\leq b$.
\end{remark}

For each $P=(c,d)\in E_B(\ZZ)$, we compute using~\eqref{eq:discred}, that the Hessian covariant of $F_P$ is
\[H(x,y)=cx^2+\frac{2(cw-d)}{M}xy+\frac{c^2+cw^2-2dw}{M^2}y^2
,\]
 which is divisible by $g_0\coloneqq\gcd(c,B)$. Therefore the Hessian of the reduced form of $F_P$ is also divisible by $g_0$, in particular the seminavariant $H$ of the reduced form will be divisible by $g_0$. 

Recall that now $\Delta(F_P)=-4kg^2$, where $g=\prod_{p\mid \gcd(c,B)} p^{v_p(B)}$.
Factor $g=g_0g_1$, where 
\[g_0\coloneqq\gcd(c,B),\quad \text{ and }\quad g_1\coloneqq \prod_{p\mid \gcd(c,B)}p^{\max\{v_p(B)-v_p(c),0\}}.\]
Notice that $g_1$ is only non-trivial if there exists some prime $p$ such that $v_p(B)> v_p(c)\geq 1$.
\begin{lemma}\label{lemma:cubefree}
Suppose $d,c,k,B$ are integers such that $d^2=c^3+kB^2$.
If $p$ is an odd prime satisfying $v_p(B)> v_p(c)\geq 1$, then $p^3\mid B$.
\end{lemma}
\begin{proof}
It is enough to show that it is impossible to have an odd prime $p$ satisfying both $v_p(B)=2$ and $v_p(c)=1$. If it is the case that $v_p(B)=2$ and $v_p(c)=1$, from the equation $d^2=c^3+kB^2$, we see that $v_p(c^3)<v_p(kB^2)$, so $2v_p(d)=v_p(d^2)=v_p(c^3)=3$, which is a contradiction.
\end{proof}
If $B$ is cube-free, by Lemma~\ref{lemma:cubefree}, the only prime that can satisfy $v_p(B)> v_p(c)\geq 1$ is $2$, so $g_1\mid 2^2$.

 By Proposition~\ref{prop:reducedbd}, the seminvariant $H$ of the reduced form satisfies 
\[
\lvert H\rvert\leq 2^{\frac{4}{3}}3^{-\frac{1}{2}}g\lvert k\rvert^{\frac{1}{2}},
\]
and the syzygy from~\eqref{eq:syzygy} becomes
\[
\left(\frac{1}{2}U\right)^2=H^3+kg^2a^2.
\]
Writing $H=hg_0$ and  $\frac{1}{2}U=ug_0$, we have
\begin{equation}\label{eq:hbound}
\lvert a\rvert\leq 2^{2}3^{-\frac{3}{4}}g^{\frac{1}{2}}\lvert k\rvert^{\frac{1}{4}},\qquad
\lvert h\rvert\leq 2^{\frac{4}{3}}3^{-\frac{1}{2}}g_1\lvert k\rvert^{\frac{1}{2}},
\end{equation}
\begin{equation}\label{eq:quadrep}
u^2-kg_1^2a^2=g_0h^3.
\end{equation}

\subsection{Proof of Theorem~\ref{theorem:ptbd}}
We first count the number of integral points when $g$ is small.

\begin{lemma}\label{lemma:smallgcd}
Let $K\leq L\leq N$ be positive numbers. Then if $k$ is not a square,
\begin{align*}
&\sum_{B\leq N}\#\left\{(c,d)\in E_B(\ZZ):
\gcd(c,B)\leq L,\ 
\prod_{p\mid \gcd(c,B)}p^{\max\{v_p(B)-v_p(c),0\}}\leq  K
\right\}\\
&\ll_k K^{4}\min\{N^{\frac{2}{3}}L^{\frac{1}{3}}\log L,\ N^{\frac{19}{21}}\log N+L\}+N^{\frac{2}{3}}(\log N)^{1+\epsilon}.
\end{align*}
\end{lemma}
\begin{proof}
Suppose $P=(c,d)\in E_B(\ZZ)$, $g_0\leq L$, and $g_1\leq  K$. If $F_P$ is reducible, $f_P$ is also reducible. By Lemma~\ref{lemma:reducible}, such $P$ contributes $\ll_k N^{\frac{2}{3}}(\log N)^{1+\epsilon}$.

Now assume that $F_P$ is irreducible.
We count the number of $\GL_2(\ZZ)$-equivalence class of $(F_P,(1,0))$. 
We can transform each $(F_P,(1,0))$ under $\GL_2(\ZZ)$ to some $(F,(x,y))$ such that $F$ has seminvariants bounded as in Proposition~\ref{prop:reducedbd}.
Notice that $F(x,y)=F_P(1,0)=B/g\leq N/g_0$.
For each irreducible integral cubic form $F$, the number of integral solutions $(x,y)$ to 
the Thue inequality 
\[\lvert F(x,y)\rvert\leq \frac{N}{g_0}\]
is $\ll (N/g_0) ^{\frac{2}{3}}$ by~\cite[Theorem~1]{Thunder}.
Therefore each $\GL_2(\ZZ)$-equivalence class of $F$ is associated to $\ll (N/g_0)^{\frac{2}{3}}$ integral points.

Observe that $a,h,u,g_0,g_1$ together determines the $\GL_2(\ZZ)$-equivalence class of $F$. Moreover $h\neq 0$ when $k$ is not a square, and so $g_0$ is determined by $(a,h,u,g_1)$ by~\eqref{eq:quadrep}.
Then
\[\sum_{B\leq N}\#\left\{(c,d)\in E_B(\ZZ):g_0\leq L,\ g_1\leq K\right\}
\ll_k\sum_{\substack{(a,h,u,g_1)\\ g_0\leq L}}\left(\frac{N}{g_0}\right)^{\frac{2}{3}}.\]
Split $g_0\leq L$ into dyadic intervals. Suppose 
$A\leq g_0<2A$, then from~\eqref{eq:hbound} and~\eqref{eq:quadrep}, we see that 
\[\lvert a\rvert \ll_k g^{\frac{1}{2}}\ll A^{\frac{1}{2}}K^{\frac{1}{2}},\qquad 
\lvert h\rvert \ll_k g_1\ll K,\]
\[\lvert u\rvert \ll_k \max\left\{\lvert g_1a\rvert ,\lvert g_0^{\frac{1}{2}}h^{\frac{3}{2}}\rvert \right\}\ll A^{\frac{1}{2}}K^{\frac{3}{2}}.\]
Now bound the number of integral points associated to tuples $(a,h,u,g_1)$ such that $A\leq g_0<2A$, we get
\[\sum_{\substack{(a,h,u,g_1)\\ A\leq g_0<2A}}\left(\frac{N}{g_0}\right)^{\frac{2}{3}}\ll_k 
\sum_{a\ll A^{\frac{1}{2}}K^{\frac{1}{2}}}
\sum_{h\ll K}
\sum_{g_1\ll K}
\sum_{u\ll A^{\frac{1}{2}}K^{\frac{3}{2}}}
\left(\frac{N}{A}\right)^{\frac{2}{3}}
\ll
N^{\frac{2}{3}}A^{\frac{1}{3}}K^{4}.\]
Therefore summing over $A$ up to $L$, we get
\[\sum_{B\leq N}\#\left\{(c,d)\in E_B(\ZZ):g_0\leq L,\ g_1\leq K\right\}\ll_k N^{\frac{2}{3}}L^{\frac{1}{3}}K^{4}\log L.\]

We can do better when $g_0$ is large. By~\cite[Theorem~6.3]{Evertse}, when $\lvert\Delta(F)\rvert\gg m^5$, the number of primitive solutions to $\lvert F(x,y)\rvert\leq m$ is $\ll 1$.
When $g_0\gg N^{\frac{5}{7}}$, we have $\lvert \Delta(F)\rvert=\lvert -4kg^2\rvert\gg N^{\frac{10}{7}}$ and $N/g_0\leq N^{\frac{2}{7}}$.
Therefore when $L\gg N^{\frac{5}{7}}$, we have
\[\sum_{B\leq N}\#\left\{(c,d)\in E_B(\ZZ):N^{\frac{5}{7}}\ll g_0\leq L,\ g_1\leq K\right\}\ll_k\sum_{\substack{(a,h,u,g_1)\\  g_0\leq L}} 1
\ll LK^{4}.
\]
\end{proof}

\begin{remark}
The assumption that $k$ is not a square  is required in Lemma~\ref{lemma:smallgcd} to exclude the reduced forms that result in $H=h=0$. Such forms would contribute extra factors of $\log L$ to the upper bound on the number of integral points when $k$ is a square.
\end{remark}

Theorem~\ref{theorem:ptbd} follows from taking $L=N$ in Lemma~\ref{lemma:smallgcd} and noting that we can take $K= 2^2$ when $B$ is cube-free by Lemma~\ref{lemma:cubefree}.

\subsection{Proof of Theorem~\ref{theorem:curvebd}}
For Theorem~\ref{theorem:curvebd}, we no longer require that $B$ is cube-free. The curves that contains points with small $g_0$ and $g_1$ will be handled by Lemma~\ref{lemma:smallgcd}.
We now want to bound the number of $B$ such that $E_B(\ZZ)$ contains a point with a large $g_1$. Since $g_1$ must divide the cube-full part of $B$ by Lemma~\ref{lemma:cubefree}, it suffices to give an upper bound on the number of $B$ with a large cube-full part. 
\begin{lemma}\label{lemma:largesq}
Let $K\leq N$ be positive integers. Then
\[\#\left\{1\leq B\leq N:  \prod_{p^3\mid B} p^{v_p(B)}\geq K \right\}\ll NK^{-\frac{2}{5}}.\]
\end{lemma}
\begin{proof}
Writing
\[B=\prod_{p^3\mid B}p^{3\lfloor\frac{1}{3}v_p(k)\rfloor}B_0,\]
we see that by assumption
$\prod_{p^3\mid B}p^{\lfloor\frac{1}{3}v_p(k)\rfloor}\geq\prod_{p^3\mid B}p^{\frac{1}{5}v_p(k)}\geq K^{\frac{1}{5}}$, so $B_0\leq NK^{-\frac{3}{5}}$. Therefore
\[
\#\Big\{1\leq B\leq N:  \prod_{p^3\mid B} p^{v_p(B)}\geq K \Big\}\ll\sum_{B_0\leq  NK^{-\frac{3}{5}}}\left(\frac{N}{B_0}\right)^{\frac{1}{3}}\ll NK^{-\frac{2}{5}}.
\]
\end{proof}
To prove Theorem~\ref{theorem:curvebd}, take $L=N(\log N)^{-\frac{11}{2}}$ and $K=(\log N)^{\frac{5}{4}}$ in Lemma~\ref{lemma:smallgcd} and Lemma~\ref{lemma:largesq}, then 
\[\#\{1\leq B\leq N: 
g_0\leq L \text{ for some }(c,d)\in E_B(\ZZ)
 \}
 \ll_k
 N(\log N)^{-\frac{1}{2}}
 .\]
It remains to count the number of curves which contain some $(c,d)\in E_B(\ZZ)$ with $g_0>L$.

\begin{lemma}\label{lemma:largegcd}
Fix $\epsilon>0$ and an integer $k\neq 0$ that is not a rational square.
Let $N$ and $L$ be positive numbers such that $N^{1-\epsilon}\leq L\leq N$.
Then
\[
\#\left\{1\leq B\leq N:\gcd(c,B)>L 
\text{ for some }(c,d)\in E_B(\ZZ)\right\}
\ll_{k,\epsilon}
N\left(\frac{\log (N/L)}{\log N}\right)^{\frac{1}{2}}.
\]
\end{lemma}
\begin{proof}
We now need to make use of the assumption that $k$ is not a square. 
Write $B=mn$, where 
\begin{align}\label{eq:m}
m&=
\prod_{p\mid \gcd(B,2k)}p^{v_p(B)}
\prod_{\substack{\text{odd }p\mid B\\ \leg{k}{p}=1}}p^{v_p(B)}
\prod_{\substack{\text{odd }p\mid B\\ \leg{k}{p}=-1}}p^{2\lfloor \frac{1}{2}v_p(B)\rfloor}\text{ and}\\\label{eq:n}
n&=\prod_{\substack{\text{odd }p\mid B\\ \leg{k}{p}=-1\\v_p(B)\text{ odd}}}p
.\end{align}
From~\eqref{eq:quadrep}, we see that if $p\mid g_0$, either $p\mid k$,  $\leg{k}{p}=1$ or $v_p(g_0)$ is even. Since $g_0\mid m$, so $m\geq g_0$. Any prime $p\nmid 2k$ that divides $m$ with odd order satisfies $\leg{k}{p}=1$, and $n$ is only divisible by primes $p$ satisfying $\leg{k}{p}=-1$.
By a result of Landau~\cite{LandauRep} (or see~\cite[Theorem~2.8] {SerreFrob} for a more general statement for sets of primes satisfying Frobenian conditions), we see that the number of integers $m\leq N$ of the form~\eqref{eq:m} is $\sim c_k N(\log N)^{-\frac{1}{2}}$, since $k\neq 1$. Similarly the number of integers $n\leq N$ of the form~\eqref{eq:n} is $\sim c'_k N(\log N)^{-\frac{1}{2}}$.
We count the number of $B\leq N$ that are divisible by some $m\geq L$.
By partial summation
\[\begin{split}
\# &\{1\leq B\leq N:g_0> L \text{ for some }(c,d)\in E_B(\ZZ) \}\\
&\ll_k\sum_{n\leq N/L} \sum_{m\leq N/n} 1
\ll \sum_{n<N/L} \frac{N}{n\sqrt{\log(N/n)}}
\ll_{\epsilon} \frac{N}{\sqrt{\log N}}\sum_{n<N/L} \frac{1}{n}\\
&\ll_k \frac{N}{\sqrt{\log N}}\left(\frac{1}{\sqrt{\log (N/L)}}+ \int_{t<N/L}\frac{dt}{t\sqrt{\log t}}\right)
\ll N\left(\frac{\log (N/L)}{\log L}\right)^{\frac{1}{2}}.
\end{split}\]
\end{proof}
Putting in $L=N(\log N)^{-\frac{11}{2}}$ in Lemma~\ref{lemma:largegcd} gives in upper bound of \[\#\{1\leq B\leq N: 
g_0>L \text{ for some }(c,d)\in E_B(\ZZ)
 \}\ll_k N\left(\frac{\log\log N}{\log N}\right)^{\frac{1}{2}}.\]
This completes the proof of Theorem~\ref{theorem:curvebd}.

\end{document}